\documentclass[pdflatex,sn-mathphys-num]{sn-jnl}
\usepackage[utf8]{inputenc}
\usepackage{amsmath,amssymb,amsfonts,amsthm}

\usepackage{pgfplots}
\pgfplotsset{compat=newest}
\pgfplotsset{plot coordinates/math parser=false}

\usepackage{hyperref}
\usepackage[capitalize,nameinlink]{cleveref}
\usepackage{dsfont}
\usepackage{todonotes}
\usepackage{multicol}
\usepackage{enumitem}
\usepackage{graphicx} 
\usepackage{mathtools}
\usepackage{xspace}
\usepackage{nicefrac}
\graphicspath{ {images/} }
\allowdisplaybreaks
\usepackage{aliascnt}

\newtheorem{theo}{Theorem}[section]

\newaliascnt{defi}{theo}
\newtheorem{defi}[defi]{Definition}
\aliascntresetthe{defi}

\newaliascnt{exa}{theo}

\aliascntresetthe{exa}

\newaliascnt{cor}{theo}

\aliascntresetthe{cor}

\newaliascnt{lem}{theo}
\newtheorem{lem}[lem]{Lemma}
\aliascntresetthe{lem}

\newaliascnt{prop}{theo}

\aliascntresetthe{prop}

\newaliascnt{ass}{theo}
\newtheorem{ass}[ass]{Assumption}
\aliascntresetthe{ass}

\newaliascnt{rem}{theo}
\newtheorem{rem}[rem]{Remark}
\aliascntresetthe{rem}

\newaliascnt{notation}{theo}
\newtheorem{notation}[notation]{Notation}
\aliascntresetthe{notation}

\crefname{theo}{Thm.}{Theorems}
\crefname{ass}{Asm.}{Assumptions}
\crefname{defi}{Def.}{Definitions}
\crefname{rem}{Rmk.}{Remarks}
\crefname{prop}{Prop.}{Propositions}
\crefname{lem}{Lem.}{Lemmata}

\usepackage{xcolor}
\newcommand{\scale}[1]{\scalebox{0.6}{\textcolor{black}{$#1$}}}

\usepackage{accents}
\newcommand{\R}{\mathbb{R}}
\newcommand{\N}{\mathbb{N}}

\newcommand{\mC}{\mathcal C}
\newcommand{\mD}{\mathcal D}

\newcommand{\mL}{\mathcal{L}}
\newcommand{\mTV}{\mathcal{TV}}
\newcommand{\mBV}{\mathcal{BV}}
\newcommand{\dd}{\ensuremath{\,\mathrm{d}}}
\newcommand{\loc}{\textnormal{loc}}
\newcommand{\mW}{\mathcal{W}}
\DeclareMathOperator*{\esssup}{ess-\sup}
\DeclareMathOperator{\supp}{supp}
\DeclareMathOperator*{\essinf}{ess-\inf}

\newcommand{\eps}{\epsilon}

\newcommand{\rtilde}[1]{\textcolor{black}{\tilde{\textcolor{black}{#1}}}}
\newcommand{\tp}{\scale{+}}
\newcommand{\tm}{\scale{-}}
\newcommand{\tk}{\scale{k}}
\newcommand{\tn}{\scale{0}}
\newcommand{\te}{\scale{1}}
\newcommand{\tz}{\scale{2}}
\newcommand{\ts}{\scale{*}}
\newcommand{\tR}{\scale{R}}
\newcommand{\tmin}{\scale{\mathrm{min}}}
\newcommand{\tmax}{\scale{\mathrm{max}}}
\newcommand{\teps}{{\scale{\eps}}}
\newcommand{\tnu}{\scale{\nu}}

\newcommand{\Wp}{\mathsf{W}_{\!_{\tp}}}
\newcommand{\Wm}{\mathsf{W}_{\!_{\tm}}}
\newcommand{\gammap}{\gamma_{\!_{\tp}}}
\newcommand{\gammam}{\gamma_{\!_{\tm}}}

\newcommand{\gammapk}{\gammap^{\tk}}
\newcommand{\gammamk}{\gammam^{\tk}}
\newcommand{\Wpk}{\Wp^{\tk}}
\newcommand{\Wmk}{\Wm^{\tk}}

\newcommand{\um}{\underline{m}}
\newcommand{\mmk}{\mathsf{M}_{^{\tm}}^{_{\tk}}}
\newcommand{\mpk}{\mathsf{M}_{^{\tp}}^{\tk}}
\renewcommand{\mp}{\mathsf{M}_{^{\tp}}}
\newcommand{\mm}{\mathsf{M}_{^{\tm}}}

\newcommand{\tqk}{\rtilde q^{\tk}}
\newcommand{\qk}{q^{\tk}}
\newcommand{\qeps}{q_{\teps}}
\newcommand{\qn}{q_{\tn}}

\newcommand{\qnR}{q_{\tn}^{\tR}}
\newcommand{\qR}{q^{\tR}}
\newcommand{\qepsR}{q_{\teps}^{\tR}}
\newcommand{\qepsnuR}{q_{\teps,\tnu}^{\tR}}
\newcommand{\QR}{Q^{\tR}}
\newcommand{\QepsR}{Q_{\teps}^{\tR}}
\newcommand{\QepsnuR}{Q_{\teps,\tnu}^{\tR}}

\newcommand{\tqn}{\rtilde q_{\tn}}
\newcommand{\qneps}{q_{\tn,\teps}}
\newcommand{\qs}{q^{\ts}}

\newcommand{\talphak}{\rtilde \alpha^{\tk}}
\newcommand{\talphas}{\rtilde \alpha^{\ts}}
\newcommand{\qkeps}{\qk_{\teps}}

\newcommand{\qmin}{q_{\tmin}}
\newcommand{\qmax}{q_{\tmax}}

\newcommand{\Qk}{Q^{\tk}}
\newcommand{\Qn}{Q_{\tn}}
\newcommand{\Qs}{Q^{\ts}}
\newcommand{\tQ}{\rtilde Q}
\newcommand{\tQk}{\rtilde Q^{\tk}}
\newcommand{\tQs}{\rtilde Q^{\ts}}
\newcommand{\tQn}{\rtilde Q_{\tn}}

\newcommand{\Qkepsk}{Q^{\tk}_{\teps_{\tk}}}

\newcommand{\Qepsp}{Q^{\teps}_{\tp}}
\newcommand{\Qepsm}{Q^{\teps}_{\tm}}
\newcommand{\Ueps}{U^{\teps}}
\newcommand{\Leps}{L^{\teps}}

\newcommand{\uUeps}{\underline{U}^{\teps}}
\newcommand{\oLeps}{\overline{L}^{\teps}}

\newcommand{\sCkm}{\mathsf C^{\tk}_{\tm}}
\newcommand{\sCkp}{\mathsf C^{\tk}_{\tp}}

\newcommand{\sCn}{\mathsf C_{\tn}}
\newcommand{\sCe}{\mathsf C_{\te}}
\newcommand{\sCz}{\mathsf C_{\tz}}
\newcommand{\sCkeps}{\mathsf C_{\tk}^{\teps}}

\newcommand{\orho}{\bar{\rho}}
\newcommand{\oq}{{\bar q}}
\newcommand{\oqk}{{\bar q_{\tk}}}
\newcommand{\HJ}{HJ\xspace }

\newenvironment{talign*}
 {\let\displaystyle\textstyle\csname align*\endcsname}
 {\endalign}

 \newenvironment{tgather*}
 {\let\displaystyle\textstyle\csname gather*\endcsname}
 {\endgather}

\title{Nonlocal Approximation Principle for Entropy Solutions of Scalar Conservation Laws}
\author[1]{Alexander Keimer}
\author[2]{Lukas Pflug}
\affil[1]{Institute of Mathematics, University of Rostock, Rostock}
\affil[2]{Department of Mathematics, Friedrich-Alexander-University (FAU) Erlangen-Nürnberg, Erlangen}
\abstract{We establish a general nonlocal approximation principle for the entropy solutions of scalar conservation laws on $\R$. More precisely, we show that the entropy solution to a nonnegative initial datum can be obtained as a weak-star limit of a corresponding scalar nonlocal conservation law. The flux function of the nonlocal conservation law depends on suitable spatial averages of the density. The proof is based on a reformulation on the Hamilton--Jacobi level: working with the primitives, we identify the limit via the stability properties of viscosity solutions; we then recover the entropy solution using the classical relation between Hamilton--Jacobi equations and scalar conservation laws. We further show that the approximation extends, after a suitable shift, to sign-changing initial data, and we prove a quantitative convergence estimate for convex fluxes in terms of the first moments of the nonlocal kernels. This result makes it possible to define entropy solutions for general fluxes using their nonlocal approximations, which satisfy the requirement for a finite speed of mass propagation, a key feature of hyperbolic conservation laws.} 
\keywords{conservation laws, singular limit, nonlocal conservation laws, approximating entropy solutions, nonlocal entropy selection principle}
\begin{document}
\maketitle

\section{Introduction and Problem Statement}
We show that scalar conservation laws on \(\R\), that is, 
\begin{equation}
\begin{aligned}
    \partial_{t}q + \partial_{x}f(q) &= 0,&& \text{ on } (0,T)\times\R, \\
    q(0,\cdot) &= \qn, && \text{ on } \R,
\end{aligned}
\label{eq:local_c_l}
\end{equation}
with \(\qn\in \mL^{\infty}(\R;\R_{\geq 0})\) and flux \(f:\R\to\R\), can be approximated by nonlocal  conservation laws. 
We thus prove that the Kruzhkov entropy solution of a scalar conservation law can be selected as the weak-star limit of a nonlocal conservation law. We remove the monotonicity condition on the induced velocity \(s\mapsto f(s)/s\) of all previous singular-limit results by replacing it with a two-sided monotone splitting at the nonlocal level. In addition, the presented results hold for all one-sided, monotone sequences of kernels approaching a Dirac distribution.

The relevant nonlocal conservation laws have the following structure:
\begin{equation}
\begin{aligned}
    \partial_{t}q + \partial_x\big(q \, V(\Wm[q],\Wp[q])\big) &= 0, &&\text{on } (0,T)\times\R,\\
    q(0,\cdot) &= \qn,&& \text{on } \R,\\
    \Wp[q](t,x)&\coloneqq \int_{x}^{\infty} \gammap(x-y)q(t,y)\dd y,&& (t,x)\in(0,T)\times\R,\\
    \Wm[q](t,x)&\coloneqq  \int_{-\infty}^{x} \gammam(x-y)q(t,y)\dd y, && (t,x)\in (0,T)\times\R.
\end{aligned}
\label{eq:nonlocal_c_l}
\end{equation}
Here, \(\Wp\) and \(\Wm\) are one-sided spatial averages of \(q\), describing
right- and left-sided interactions, respectively. The velocity \(V\) is chosen
so that the local compatibility condition \(sV(s,s)=f(s)\) is satisfied. Motivated by the positive/negative wave-speed decomposition underlying the
Engquist--Osher flux splitting \cite{EngquistOsher1980}, we use a similar
splitting of the induced local velocity \(s\mapsto f(s)/s\) which can for instance consist of the decomposition into the monotonically increasing and decreasing parts,\footnote{For other choices and more detailed comments on \(V\), see \cref{rem:V_f_compatibility}.} as follows:
\begin{equation}
\R^{2}\ni(a,b)\mapsto V(a,b)\coloneqq \tilde V(0) + \int_{0}^{a}\max\{\tilde{V}'(s),0\}\dd s+\int_{0}^{b}\min\{\tilde{V}'(s),0\}\dd s.
\label{eq:nonlocal_velocity_natural}    
\end{equation}
Denoting by \(q_{\gammap,\gammam}\) the (unique) weak solution of \cref{eq:nonlocal_c_l}, we will show under rather weak assumptions on the flux, initial datum, and nonlocal kernels that \[
q_{\gammap,\gammam}\xrightharpoonup[\gammam \rightarrow \delta_0 \text{ in } \mD'(\R)]{\gammap \rightarrow \delta_0 \text{ in } \mD'(\R)} q \ \text{ in } \ \mL^{\infty}((0,T)\times\R),
\]
where \(q \in \mL^{\infty}((0,T)\times\R)\) is the local entropy solution of \cref{eq:local_c_l}.

Since nonlocal conservation laws of the form \cref{eq:nonlocal_c_l} admit
unique weak solutions \cite{pflug,keimer2022discontinuous}, the local entropy
solution can be obtained as the singular limit of these nonlocal
approximations. This provides a selection principle analogous to the
vanishing-viscosity method with the advantage that the nonlocal ``regularization'' conserves the finite speed of density propagation (whereas information transfer via nonlocal averages is instantaneous).

The result is also relevant for control and optimal-control problems; see
\cite{friedrich2025control,keimer2025optimal} for related singular-limit
results. Nonlocal regularization provides stronger analytical properties,
such as differentiability of the control-to-state map or controllability
features \cite{bayen2021boundary}. More generally, for many applications, reasonable models may indeed be nonlocal, and the local version may be a formal simplification. In these contexts, the convergence result implies that the dynamics are robust to nonlocal disturbances.

\subsection{Theory of Nonlocal Scalar Conservation Laws and the Singular Limit} Given the (local) conservation law  in \cref{eq:local_c_l}, there are several structurally different nonlocal approximations.
\textbf{The first class} acts directly on the flux:
\begin{align}\label{eqn:flux_filtering}
\partial_{t}q+\partial_{x}\big(\gamma*f(q)\big)=0, \qquad \partial_t q + \partial_x \big( f(\gamma*q ) \big) = 0
\end{align}
where \(\ast\) denotes spatial convolution and \(\gamma\) is a nonlocal kernel. These ``flux-filtering''approximations were introduced and their singular limits studied in \cite{CocliteKarlsenRisebro2025}. Since the conservative flux is nonlocal in these approximations, the corresponding solutions generally do not preserve finite-speed propagation of mass, similarly to viscous regularizations.

\textbf{A second class} is based on a factorization
\(f(s)=sV(s), s\in\R,\)  whenever possible, with a Lipschitz velocity \(V\), leading to
\begin{align}
    \partial_t q + \partial_x \big(q \, (\gamma*V(q))\big) = 0, \qquad \partial_{t}q+ \partial_x \big(q V(\gamma*q)\big) = 0.\label{eq:nonlocal_real}
\end{align} 
The first of these equations, which is the subject of this contribution, is obtained by introducing two one-sided kernels into the nonlocal approximation in \cref{eq:nonlocal_c_l}. Here, the convolution regularizes the velocity field, yielding well-posed transport dynamics in the spirit of DiPerna and Lions \cite{DiPerna1989,pflug}. Moreover, mass is still transported through the local density factor \(q\), so the finite-speed propagation of mass is retained.

Early numerical studies on the singular limit can be found in \cite{amorim2015numerical,zumbrun,blandin2015well,goatin2016well}. For monotone downstream kernels and decreasing velocities, convergence to the local entropy solution was first proven under additional monotonicity assumptions on the initial datum in \cite{keimer2019approximation}.
For general kernels and velocities, it was shown in \cite{colombo2019singular} that strong $\mL^1$-convergence may fail and that only weak convergence, or no convergence in certain cases, can be expected. The ``positive effect'' of adding viscosity was investigated in \cite{colombo2019singular,coclite2021singular,
colombo2021role}. 

Subsequent progress relied on more refined compactness mechanisms: \citet{bressan2020traffic} used a reformulation as a system in a specific setting, \citet{crippa2021local} obtained convergence under a one-sided Lipschitz condition despite the lack of uniform \(\mTV\) bounds, and \citet{bressan2021entropy} established entropy admissibility for strong \(\mL^1\)-limits in the one-sided exponential case. The first general result was obtained in \cite{coclite2022general} using compactness of the nonlocal quantity rather than of the solutions themselves.

Further extensions have covered more general kernels and compactness estimates \cite{colombo2023nonlocal,colombo2024overview,keimer2025singular}, spatially discontinuous nonlocal laws \cite{keimer2023singular}, and balance laws \cite{Chiarello2024singular} or have focused directly on the solution rather than the nonlocal quantity \cite{denitti2025volterra}. 
The \(\mTV\) assumption on the initial datum was recently removed in \cite{coclite2025singular} via compensated compactness \cite{Tartar1979CompensatedCompactness}, yielding convergence for merely \(\mL^\infty\)-data. Ole\u{\i}nik-type estimates \cite{oleinik,oleinik_english} for such nonlocal approximations were presented in \cite{coclite2024oleinik} under restrictive assumptions on the data, nonlinear velocity, and kernel.
Regarding numerical schemes that are robust under the singular limit, we refer to \cite{huang2024asymptotic,denitti2025numerical}. 

For related problems on nonlocal conservation laws in which the nonlocality is in space and time simultaneously, as well as the related singular limit, we refer to \cite{du2023spacetime}. For results regarding cases in which the integral operator is of \(\mL^{p}\) type, not \(\mL^{1}\), and the corresponding singular limit, we refer to \cite{chiarello2026nonlocal}.

\textbf{A third class} of nonlocal approximations is obtained by starting from the nonconservative form and regularizing the velocity of the underlying transport equation:
\begin{align}\label{eqn:transport}
     \partial_t q + \big(\gamma * f'(q)\big)\,\partial_x q = 0,
     \qquad
     \partial_t q + f'(\gamma*q)\,\partial_x q = 0.
\end{align}
The convolution yields a Lipschitz velocity field and hence well-posed transport dynamics in the sense of \cite{DiPerna1989,pflug,Coron2020}. Moreover, uniform \(\mTV\)-bounds imply strong \(\mL^1_{\loc}\)-compactness and convergence to a weak solution of the local equation. This singular limit was studied in \cite{wiedemann2025non,Norgard2008,RiesebroWeber2014,Coron2020}. However, entropy convergence requires strong structural restrictions, in particular, symmetric kernels and quadratic fluxes. Otherwise, even mass conservation may fail in the limit. This further motivates the conservative nonlocal velocity approximation considered in \cref{eq:nonlocal_c_l}.

Among these three classes,  the nonlocal conservation laws in \cref{eq:nonlocal_real} are the only ones that simultaneously retain mass conservation and finite-speed propagation of mass at the approximate level. In contrast to the flux-filtering approximations in \cref{eqn:flux_filtering}, the conservative transport of mass remains localized through the factor \(q\); unlike in the nonconservative approximations in \cref{eqn:transport}, mass is preserved. As a generalization of \cref{eq:nonlocal_real}, \cref{eq:nonlocal_c_l} is, therefore, particularly well suited for singular-limit approximations that are intended to preserve the hyperbolic character of the original conservation law.

\subsection{Outline of the Manuscript}
We prove in \cref{sec:singular_limit} that whenever \(\gammap,\gammam\rightarrow \delta_0\) in $\mD'(\R)$, that is, in the singular limit case of \cref{eq:nonlocal_c_l}, we recover the unique entropy solution of the local conservation law in \cref{eq:local_c_l}; the underlying convergence is weak-star in \(\mL^{\infty}((0,T)\times\R)\) (see \cref{theo:singular_limit_convergence}). This is achieved by reinterpreting the local entropy solution as the spatial derivative of the viscosity solution of the corresponding Hamilton--Jacobi (\HJ) equation.
In \cref{subsec:sign_restricted}, we demonstrate how the convergence result can also be used to approximate local conservation laws with sign-unrestricted datum nonlocally. Finally, \cref{sec:convergence_order} establishes, in the case of convex or concave fluxes \(f\), the order of convergence in terms of the first moments of the kernels in \cref{theo:convergence_order}.
To achieve these results, we study in \cref{sec:nonlocal_existence_uniqueness} the existence and uniqueness of the corresponding nonlocal approximation together with a suitable stability result, and in \cref{sec:HJ}, we introduce the required techniques for Hamilton--Jacobi equations and establish the relation to the entropy solution of the corresponding conservation law.

\section{Existence and Uniqueness for Local and Nonlocal Conservation Laws}\label{sec:nonlocal_existence_uniqueness}
First, we state the well-known existence and uniqueness result for entropy solutions of scalar conservation laws, as defined by \cref{eq:local_c_l}.
\begin{theo}[existence and uniqueness of entropy solutions]\label{thm:ex_uni_c_l}
Given \(f\in \mW^{1,\infty}_{\loc}(\R)\) and \(\qn\in\mL^{\infty}(\R)\), there exists a unique weak entropy solution \[q\in\mC\big([0,T];\mL^{1}_{\loc}(\R)\big)\cap \mL^{\infty}((0,T);\mL^{\infty}(\R))\] of the Cauchy problem in \cref{eq:local_c_l} in the sense of entropy admissibility, as given, for instance, in \cite{bressan}. Letting $\tilde q$ denote the solution with initial datum $\tqn \in \mL^\infty(\R)$, with $K\subset \R$ and $L \coloneqq \|f'\|_{\mL^\infty((\qmin,\qmax))}$, we obtain the following stability estimate:
\begin{align}
    \|q-\tilde q\|_{\mL^\infty((0,T);\mL^1(K))} \leq \|\qn- \tqn\|_{\mL^1(K+(-TL,TL))}.
\end{align}
\end{theo}
\begin{proof}
    The proof can be found, for different types of entropy flux pairs, in \cite[Thm.~6.3]{bressan}, \cite[Thm.~19.1]{eymard}, \cite[Thm.~2, Thm.~5, Sec.~5 Item 4]{kruzkov}, \cite[Def.~3.1, Thm.~5.2]{godlewski} and \cite[Thm.~6.2.2]{Dafermos2016} (particularly for only Lipschitz fluxes as proposed here) and the stability result with cone of dependency, for example, in \cite[Thm.~6.3.2]{Dafermos2016}.
\end{proof}

As we aim for nonlocal approximation, we will require the existence and uniqueness of solutions to nonlocal conservation laws. We will supplement this with a stability result, which will later enable us to approximate solutions by sequences of smooth solutions.
\begin{theo}[existence/uniqueness/stability of weak solutions (nonlocal)]\label{theo:existence_uniqueness_nonlocal}
Suppose we are given \(\qn  \in \mL^{\infty}(\R)\), \(V \in \mW^{1,\infty}_\loc(\R^2)\) with $\partial_1 V \geqq 0 \wedge \partial_2 V \leqq 0$, and $(\gammam,\gammap) \in \mBV(\R_{>0};\R_{\geq 0}) \times \mBV(\R_{<0};\R_{\geq 0})$ with $\gammam$ and $\gammap$ monotonically decreasing and increasing, respectively. Then, there exists a \textbf{unique} weak solution
\begin{equation}
q\in \mC\big([0,T];\mL^{1}_{\loc}(\R)\big) \cap \mL^\infty\big((0,T)\times \R\big)\label{eq:nonlocal_solution_regularity_weak}
\end{equation}
to \cref{eq:nonlocal_c_l}, which satisfies the following statements:
\begin{enumerate}
\item \label{item:1}
 \(\forall t\in[0,T]\),
\(
    \qmin\coloneqq \essinf_{x\in\R} \qn(x)\leqq q(t,\cdot) \leqq \esssup_{x\in\R} \qn(x)\eqqcolon \qmax    \ \text{ a.e.},
\)
    \item \label{item:2} \(\qn\in\mC^{\infty}(\R)\ \implies\ q\in \mC^{\infty}([0,T]\times\R)\),
    \item \label{item:3} for $\qn \in \mL^\infty(\R)$ and $\qneps  \coloneqq \qn * \phi_{\teps}$,
\begin{align}
    \lim_{\eps \rightarrow 0 }\Big\|(t,x) \mapsto \textstyle\int_0^x q(t,y) - \qeps(t,y) \dd y\Big\|_{\mL_{\loc}^{\infty}((0,T)\times \R)} = 0.
     \label{eq:stability_nonlocal}
    \end{align}
\end{enumerate}
\end{theo}
\begin{proof}
   We begin with the existence and uniqueness part. Because we are dealing with two nonlocal operators, not just one as is common in nonlocal conservation laws in the literature \cite{pflug,keimer2022discontinuous}, existence and uniqueness of weak solutions follow after minor modification by means of a fixed-point argument in both nonlocal operators on a small time horizon. The remaining parts, including the stability claim in \cref{eq:stability_nonlocal}, are generalizations of the stability results obtained in \cite[Thm.~28]{keimer2022discontinuous}. Indeed, the statement \cite[Eq.~(78)]{keimer2022discontinuous}
implies, by means of approximation arguments in \(\mL^{1}(\R)\), that
   \[
\forall y\in \R \ : \ \lim_{\eps\rightarrow 0}\sup_{t\in[0,T]}\bigg|\int_0^y\qeps(t,x)-q(t,x)\dd x \bigg|=0.
   \]   
Thus, if we choose \(y\in K\), where \(K\subset\R\) is compact, we obtain uniform convergence because \([0,T]\times\R\ni (t,x)\mapsto \int_{0}^{x}\qeps(t,y)-q(t,y)\dd y\) is uniformly Lipschitz continuous, implying the stability result in \cref{eq:stability_nonlocal}.
\end{proof}

\section{Hamilton--Jacobi Equations and Their Relation to Conservation Laws}\label{sec:HJ}
In the following, for \(f:\R\rightarrow\R\) and \(\qn:\R\rightarrow\R_{\geq0}\) as in \cref{eq:local_c_l}, we consider the associated \HJ equation, which takes the form
\begin{equation}
    \begin{aligned}
        \partial_{t}Q+f(\partial_{x}Q)&=0, && (t,x)\in(0,T)\times\R,\\
        Q(0,x)&=\int_{0}^{x}\qn(y)\dd y,&& x\in\R.
    \end{aligned}
   \label{eq:H_J}
\end{equation}
For sub-/supersolutions and the corresponding viscosity solutions, we use the definition given in \cite{evans}.
\begin{defi}[sub-/supersolutions for \HJ equations]\label{defi:HJ_viscosity}
    For \(\qn  \in \mL^{\infty}(\R)\), we call a function \(Q\in\mC([0,T]\times\R)\) a \textbf{subsolution} to \cref{eq:H_J} iff \(Q(0,\cdot)\equiv \int_{0}^{\cdot}\qn (y)\dd y\) and it holds for all
\(\Phi\in \mC^{\infty}((0,T)\times\R)\) that
if \(Q-\Phi\)  is locally maximal at \((t_{0},x_{0})\in (0,T)\times\R\), then \(\partial_{1}\Phi(t_{0},x_{0})+f(\partial_{2}\Phi(t_{0},x_{0}))\leq 0.\)

Similarly, for \(\qn  \in \mL^{\infty}(\R)\), we call a function \(\tQ\in\mC([0,T]\times\R)\) a \textbf{supersolution} to \cref{eq:H_J} iff \(\tQ(0,\cdot)\equiv \int_{0}^{\cdot}\qn (y)\dd y\) and it holds for all
\(\Phi\in \mC^{\infty}((0,T)\times\R)\) that
if \(\tQ-\Phi\) is locally minimal at \((t_{0},x_{0})\in (0,T)\times\R\), then \(\partial_{1}\Phi(t_{0},x_{0})+f(\partial_{2}\Phi(t_{0},x_{0}))\geq 0\).

Finally, a function \(Q\in\mC([0,T]\times\R)\) is called a \textbf{viscosity solution} iff it is both a supersolution and a subsolution to \cref{eq:H_J}.
\end{defi}
\begin{theo}[existence, uniqueness, and stability of viscosity solutions for \HJ equations] \label{thm:ex_uni_HJ}
Given \(\qn  \in \mL^{\infty}(\R)\) and \(f \in \mW^{1,\infty}_\loc(\R)\), there exists a unique viscosity solution \(Q\in \mathcal{UC}((0,T)\times \R)\) of \cref{eq:H_J} in the sense of \cref{defi:HJ_viscosity}. 

Furthermore, given another initial datum \(\tQn\coloneqq \int_{0}^{\cdot}\tqn(y)\dd y\)  with \(\tqn\in\mL^{\infty}(\R)\), for which the corresponding solution is \(\tQ\in \mathcal{UC}((0,T)\times \R)\), the following stability and cone of dependency result holds for all $K\subset \R$ with $L\coloneqq \|f'\|_{\mL^{\infty}((\qmin,\qmax))}$:
\begin{equation}
    \|Q-\tQ\|_{\mL^{\infty}((0,T)\times K)}\leq \|\Qn -\tQn\|_{\mL^{\infty}(K + (-LT,LT))},
\end{equation}
assuming that \(\Qn -\tQn\in\mL^{\infty}(\R)\).
In particular, it holds that
\(Q\in \mW^{1,\infty}_\loc((0,T)\times\R)\), with \(\partial_1 Q, \partial_2 Q \in \mL^\infty((0,T)\times \R)\). 
\end{theo}
\begin{proof}
    The existence proof can be found in \cite[Thm.~1]{Crandall1986} for \(f\in \mC(\R)\) and \(\qn ,\tQn\in \mathcal{UC}(\R)\), which is satisfied by the given Lipschitz-continuous initial datum. The results on Lipschitz regularity, uniqueness, stability, and the cone of dependency result follow from the comparison principle; see \cite[Thm.~1]{Crandall1986} and \cite[Thm.~14.1]{Lions1982}.
\end{proof}

There is a well-known relation between solutions to \HJ equations and entropy solutions to local conservation laws, which we will state in the following solely for an initial datum in \(\mL^{\infty}\) (on the level of conservation laws).
\begin{theo}[relation between \HJ equations and conservation laws]\label{theo:rel_HJ_cl}
    Let \(\qn \in\mL^\infty(\R;\R_{\geq0})\) and \(f\in \mW^{1,\infty}_\loc(\R)\) be given. Let \(q\in\mL^{\infty}((0,T)\times\R)\) be the entropy solution of \cref{eq:local_c_l} with initial datum \(\qn \), and let \(Q\in\mW_\loc^{1,\infty}((0,T)\times\R)\) be the unique viscosity solution to the \HJ equation in \cref{eq:H_J} with initial datum \(\Qn \coloneqq \int_{0}^{\cdot}\qn (y)\dd y.\) Then, it holds that
    \[
        \partial_{2}Q\equiv q \text{ in } \mL^{\infty}((0,T)\times\R).
    \]
\end{theo}
\begin{proof}
    This is sketched in \cite[Thm.~1.1]{karlsen}. We, however, give a proof for the more general setup as this is crucial for our later analysis. Let \(\phi\in\mC^{1}_{\text{c}}((-42,T)\times\R)\) with \(R \in \R_{>0}\) such that $\supp(\phi) \subset [-R,R]$ be arbitrary but fixed. Let $\qnR :\equiv \qn \chi_{_{(-R-LT,R+LT)}}$ on \(\R\) and $L\coloneqq \|f'\|_{L^{\infty}((\qmin,\qmax))}$.
    Consider---with $\phi_{\teps} \in \mC^{\infty}(\R)$ for \(\eps\in\R_{>0}\) the standard mollifier---the two viscous problems
    \begin{align*}
        \partial_{1}q+\partial_{2}f(q)&=\nu \partial_{2}^{2}q & \partial_{1}Q+f(\partial_{2}Q)&=\nu \partial_{2}^{2}Q && \text{on } (0,T)\times\R\\
        q(0,\cdot)&\equiv \qnR * \phi_\teps &  Q(0,\cdot)&= \int_{0}^{\cdot}\Big(\qnR * \phi_\teps\Big)(y)\dd y &&\text{on }\R.
\end{align*}
    For \(\eps,\nu \in\R_{>0}\), there exist, by standard parabolic theory, unique classical solutions \(\qepsnuR,\QepsnuR \in \mC^{\infty}([0,T]\times \R)\). By \cref{thm:ex_uni_c_l} and \cref{thm:ex_uni_HJ},
    there exist for $\nu = 0$ two unique entropy/viscosity solutions, that is,  \(\qepsR\) and \(\QepsR\) with smooth and truncated initial datum $(\qnR * \phi_\teps)(y)$ and its primitive for the corresponding equations. By the same results, there exist two unique entropy/viscosity solutions $\qR$ and $\QR$  with unregularized but truncated initial datum $\qnR$ and its primitive for the respective equations. 
    By construction, we have the following identity:
    \begin{align}\label{eqn:HJ_CL_relation}
        \QepsnuR(t,x) = \int_{-\infty}^x \qepsnuR(t,y) \dd y - \int_{-\infty}^0 (\qnR * \phi_\teps)(y)\dd y, 
    \end{align}
   where the last integral, which has a constant value, compensates for $0$ as the lower bound in the initial datum. Thus, we obtain 
    \begin{align}\label{eqn:smooth_CL_HJ}
        \partial_{x}\QepsnuR(t,x)=\qepsnuR(t,x), \qquad \forall (t,x)\in (0,T) \times \R.
    \end{align}
By the cone of dependency property for conservation laws and \HJ equations (\cite[Chap.~III, Rmk.~3.13]{bardi2009optimal}), we know that
\begin{align}
    Q &\equiv \QR\ \quad \text{ and } \quad \ q \equiv \qR \qquad \text{ on } (-R,R).
\end{align}
To show that $\partial_2 Q$ and $q$ coincide, we take advantage of the previously stated stability estimate in \cref{thm:ex_uni_HJ} and the equivalence for smooth solutions in \cref{eqn:smooth_CL_HJ}. 

To this end, let \(\phi\in \mC^{\infty}_{\text{c}}(\R)\), and estimate for any \(t\in (0,T)\) the difference between \(\partial_{2}Q\) and \(q\) 
distributionally as follows, with $K\coloneqq (-R,R) \supset \supp \phi$:
    \begin{align}
        &\bigg|\int_\R \phi'(x) Q(t,x) +  \phi(x) q(t,x) \dd x\bigg|\\
        &\leq \|\phi'\|_{\mL^1(\R)}\Big(\big\|\QR-\QepsR \big\|_{\mL^\infty((0,T)\times K)} +\big\|\QepsR -\QepsnuR\big\|_{\mL^\infty((0,T)\times K)}  \Big) \\
        &\quad + \|\phi\|_{\mL^\infty(\R)}\Big(\big\|\qR-\qepsR\big\|_{\mL^1((0,T)\times K)}+\big\|\qepsR-\qepsnuR\big\|_{\mL^1((0,T)\times K)}\Big) \\
        &\quad + \int_\R \phi'(x) \QepsnuR(t,x) +  \phi(x) \qepsnuR(t,x) \dd x
    \end{align}
where we have added several ``zeros.''
The latter integral on the right-hand side vanishes by \cref{eqn:smooth_CL_HJ}. 
The terms
$\|\QR-\QepsR \|_{\mL^\infty((0,T)\times K)}$ vanish for $\eps \rightarrow 0$ by the stability result in \cref{thm:ex_uni_HJ}. As \cref{eqn:HJ_CL_relation} also holds for $\nu=0$, we can estimate the difference of the integrands in $\big\|\QepsR -\QepsnuR\big\|_{\mL^\infty((0,T)\times K)}$ in terms of its $\mL^1$-norm, that is,
\begin{align}
    \big\|\QepsR-\QepsnuR\big\|_{\mL^\infty((0,T)\times K)}  \leq \big\|\qepsR-\qepsnuR\big\|_{\mL^1((0,T)\times K)},
\end{align}
which vanishes by \cite[Thm.~6.3.1]{Dafermos2016}, as does $\|\qepsR-\qepsnuR\|_{\mL^1((0,T)\times K)}$. Finally, $\|\qR-\qepsR\|_{\mL^1((0,T)\times K)}$ vanishes by \cite[Thm.~6.3.2]{Dafermos2016}. Thus, we have $\partial_2 Q \equiv q$ in a distributional sense. As \(Q\in \mW^{1,\infty}_\loc((0,T)\times \R)\) with \(\partial_2 Q \in \mL^{\infty}((0,T) \times \R)\) and \(q\in \mL^{\infty}(\R)\), this identity holds a.e.
\end{proof}

\section{The Singular Limit Problem}\label{sec:singular_limit}
In this section, by taking advantage of the \HJ equations and their relation to conservation laws from \cref{theo:rel_HJ_cl}, we will demonstrate that if the nonlocal kernels \(\gammap\) and \(\gammam\) in \cref{eq:nonlocal_c_l} both approach a Dirac distribution (simultaneously, but not necessarily at the same ``speed''), then the solution to  \cref{eq:nonlocal_c_l} converges to the local entropy solution of \cref{eq:local_c_l}.

To formulate our result in precise terms, we will need to parametrize our nonlocal kernels and make the solution depend on them. This is done in the following.

\begin{ass}[singular limit convergence: nonlocal kernels and velocities]\label{ass:general}
Assume \(\qn \in\mL^{\infty}(\R;\R_{\geq0})\) and the following:
\begin{enumerate}
\item Nonlocal kernels:
\((\gammapk,\gammamk)_{k\in \N}\subset \mL^{\infty}(\R_{<0};\R_{\geq 0})\times\mL^{\infty}(\R_{>0};\R_{\geq 0})\)
T
such that
\begin{itemize}
\item \( \|\gammamk\|_{\mL^{1}(\R_{>0})}=1\),\ \(\|\gammapk\|_{\mL^{1}(\R_{<0})}=1\) \(\forall k\in\N\)
\item \(\gammamk\) and \(\gammapk\) are monotonically decreasing and increasing on \(\R_{>0}\) and \(\R_{<0}\), respectively,
 \(\forall k\in\N\)
\item \(\gammamk,\gammapk \xrightarrow{k\rightarrow\infty} \delta_0,  \text{ in } \mD'(\R)\)
\end{itemize}

\item Flux: \(f\in \mW^{1,\infty}_{\loc}(\R)\), without loss of generality \(f(0)=0\), and
\[
    s \mapsto f(s)/s, \quad  s\in \R_{\neq 0}
\]
admits a locally Lipschitz extension to \(\R\).

\item Nonlocal velocity \(V\in\mW^{1,\infty}_{\loc}(\R^{2})\) with  \(\partial_{1}V\geqq0,\ \partial_{2}V\leqq 0\), and ``compatibility'' with the flux, that is,
\begin{equation}
    sV(s,s)=f(s)\quad \forall s\in\R.\label{eq:flux_velocity_compatibility}
\end{equation}
\end{enumerate}
\end{ass}

\begin{rem}[``compatibility'' of \(V\) and \(f\)]\label{rem:V_f_compatibility}
Because we aim to approximate entropy solutions nonlocally, we require that \(V\) contains, in some sense, the structure of \(f\). This is guaranteed by \cref{eq:flux_velocity_compatibility}. This condition is no restriction on \(f\) as we can choose \(V\) as a ``natural'' and somewhat ``symmetric'' Engquist--Osher flux splitting \cite{EngquistOsher1980},  as in \cref{eq:nonlocal_velocity_natural}. An even simpler choice would be
\begin{equation}
V(a,b)=\tilde{V}\big(\tfrac{a+b}{2}\big)+\kappa (a-b),\ (a,b)\in [\qmin,\qmax]^{2}
\end{equation}
with \(\tilde{V}(s)\coloneqq \nicefrac{f(s)}{s},\ s\in\R\), and \(\kappa\in\R,\ \kappa\geq \|\tilde V'\|_{\mL^{\infty}((\qmin,\qmax))}\).

Note also that the condition on the flux in \cref{item:2} of \cref{ass:general} is satisfied if we assume that \(f|_{(-\eps,\eps)}\in \mW^{2,\infty}((-\eps,\eps))\) for some \(\eps\in\R_{>0}\).
Note also that \cref{item:2} of \cref{ass:general} means that the velocity of the conservation law is finite at zero density (i.e., speed is finite in the vacuum case).
\end{rem}

\begin{notation}[abbreviations]\label{notation}
    For the nonlocal kernels \((\gammamk,\gammapk)_{k\in \N}\) as in \cref{ass:general},  for all $k\in \N$, we denote the unique weak solution to \cref{eq:nonlocal_c_l} by \(\qk\in \mC\big([0,T];\mL^{1}_{\loc}(\R;\R_{\geq 0})\big)\cap\mL^{\infty}((0,T);\mL^{\infty}(\R_{\geq 0}))\). In addition, $\forall  (t,x)\in(0,T)\times\R$, we define $\Wpk[\qk](t,x)\coloneqq \int_{x}^{\infty} \gammapk(x-y)\qk(t,y)\dd y$
    and $\Wmk[\qk](t,x)\coloneqq \int_{-\infty}^{x} \gammamk(x-y)\qk(t,y)\dd y.$  
\end{notation}
\begin{lem}[properties satisfied by the kernels]\label{lem:kernel_vanishing_first_moments}
For the kernel class considered in \cref{ass:general}, it holds that $\forall x\in \R_{>0}$,
\begin{align}
          &\lim_{k\rightarrow\infty} \|  \gammamk\|_{\mL^{1}((x,\infty))} = 0
        \hspace{-1cm}&\text{and} \qquad &\lim_{k\rightarrow\infty}  \|  \gammapk\|_{\mL^{1}((-\infty,-x))} = 0,
\label{eq:kernel_vanishing_tail} \\
       &\lim_{k\rightarrow\infty} \textstyle\int_{0}^{x}  y\gammamk(y) \dd y  = 0  \hspace{-1cm}&\text{and} \qquad &\lim_{k\rightarrow\infty} \textstyle\int_{0}^{x} y\gammapk(-y) \dd y = 0,\\
   &\lim_{k\rightarrow \infty} \gammamk\big(x^\pm\big) = 0  \hspace{-1cm}&\text{and} \qquad  &\lim_{k\rightarrow \infty} \gammapk\big(-x^\pm\big) = 0,
\end{align}
where, for example, $\gammamk(x^\pm)$ refers to the left- and right-sided limit of $\gammamk$ at $x\in \R_{>0}$.
\end{lem}
\begin{proof}
The proof of the claimed limits is a direct consequence of the distributional convergence of $\gammamk$ and $\gammapk$ to $\delta_0$ in a distributional sense.
\end{proof}
Now, we prove our main theorem.
\begin{theo}[nonlocal approximation of entropy solutions]\label{theo:singular_limit_convergence}
Let \cref{ass:general} hold. Then, the nonlocal solution \((\qk)_{k\in\N}\subset \mC\big([0,T];\mL^{1}_{\loc}(\R)\big)\cap\mL^{\infty}((0,T);\mL^{\infty}(\R))\)  as introduced in \cref{notation} and the related nonlocal terms converge weak-star to the entropy solution \(q \in \mL^{\infty}((0,T)\times\R)\) of the local conservation law \cref{eq:local_c_l}, that is,
\begin{equation}
\qk,\, \Wpk[\qk],\, \Wmk[\qk]\xrightharpoonup[k\rightarrow \infty]{*} q \text{ in } \mL^{\infty}((0,T)\times\R).\label{eq:weak_star:convergence}
\end{equation}
\end{theo}
\begin{proof}
The proof consists of reformulating this convergence on the \HJ level, passing to the limit by taking advantage of sub- and supersolutions for \HJ equations as in \cref{defi:HJ_viscosity}, and using the relation between viscosity solutions of \HJ equations and entropy solutions for (local) scalar conservation laws as stated in \cref{theo:rel_HJ_cl}.

\begin{enumerate}[label=(\alph*),resume]

\item \textit{Smooth approximation and compactness:}
For each $\varepsilon \in \R_{>0}$, we denote by \(\qkeps \in \mC^{\infty}([0,T]\times\R)\) the solution to \cref{eq:nonlocal_c_l} for the initial datum \(\qneps \coloneqq \qn  \ast \phi_{\teps}\). By the smoothness of the initial datum and the nonlocal regularity property in  \cref{item:2} of \cref{theo:existence_uniqueness_nonlocal}, this is a classical solution.

By \cref{item:3} of \cref{theo:existence_uniqueness_nonlocal}, we can choose for every $k\in \N$, an $\eps_{\tk} \in \R_{>0}$ small enough that
\begin{align}\label{eqn:subsequence1k}
    \bigg\|(t,x) \mapsto \int_0^x \qkeps(t,y) -  \qk(t,y) \dd y\bigg\|_{\mL^{\infty}((0,T)\times (-k,k))} \leq \tfrac{1}{k}.
\end{align}
For convenience, we define \(\tqk :\equiv \qkeps\) and
    \begin{align}\label{eqn:defi_tQk}
\tQk  (t,x)&\coloneqq\int_{0}^{x}\tqk (t,y)\dd y - \talphak(t),
\end{align}
for \((t,x)\in (0,T)\times\R\), with 
\begin{align}
\forall t\in [0,T]  \quad \talphak(t) \coloneqq \int_0^t  \tqk(s,0) V\big(\Wmk[\tqk](s,0),\Wpk[\tqk](s,0)\big) \dd s.\label{eq:defi_alpha_k}
\end{align}
Here, the $\talphak$ term is introduced so that $\tqk$ satisfies the nonlocal \HJ equation, that is,
\begin{align} \label{eqn:nonlocal_HJ_solution}
    \partial_1 \tQk + \partial_2 \tQk V\big(\Wmk[\partial_2 \tQk],\Wpk[\partial_2 \tQk]\big) &= 0.
\end{align}

Indeed, by differentiating $\tQk$  with respect to $t\in[0,T]$, using the fact that $\tqk$ is a classical solution of \cref{eq:nonlocal_c_l}, and integrating by parts, we obtain \cref{eqn:nonlocal_HJ_solution}. In particular, thanks to \cref{eqn:defi_tQk}, it holds that
\begin{align}\label{eqn:Qx_q_equivalence}
    \partial_2 \tQk \equiv \tqk \text{ on } [0,T]\times\R,
\end{align}
which is also required to show \cref{eqn:nonlocal_HJ_solution}.

In addition, we have by \cref{eq:defi_alpha_k} that $\talphak(0) = 0$, and thus by uniform convergence, $\talphas(0) = 0$. From this and \cref{eqn:defi_tQk}, we obtain $\tQk(0,x) = \int_0^x \tqk (0,y) \dd y \ \forall k$ and---again by uniform convergence---$\tQs(0,x) = \int_0^x \qn (y) \dd y$.

Since $(\tqk)_{k\in \N}$ is a sequence of classical solutions of \cref{eq:nonlocal_c_l} and satisfies the uniform maximum principle, we find for every $k\in\N$ and every compact \(K\subset\R\) that 
    \begin{equation}
    \begin{aligned}
    \|\tQk\|_{\mL^\infty((0,T)\times K)} &\leq \sup \{|x| : x\in K\} \, \|\qn \|_{\mL^\infty(\R)} + T\|\qn \|_{\mL^\infty(\R)}\|V\|_{\mL^\infty((\qmin,\qmax)^2)}\hspace{-.7cm}\\
   \| \partial_{1}\tQk\|_{\mL^{\infty}((0,T)\times \R)} &\leq \|\qn \|_{\mL^{\infty}(\R)}\|V\|_{\mL^\infty((0,\qmax)^2)}\\
\|\partial_{2}\tQk\|_{\mL^{\infty}((0,T)\times\R)}&\leq \|\qn \|_{\mL^{\infty}(\R)}.
    \end{aligned}
    \label{eq:Lipschitz_bounds_tQk}
    \end{equation}
Hence, by the Ascoli--Arzel\`{a} theorem \cite[Thm.~1.33]{adams2003sobolev}, there exists \(\tQs \) and a subsequence, again denoted by $k$, such that \(\lim_{k\rightarrow \infty} \|\tQk-\tQs \|_{\mL^{\infty}((0,T)\times K)}\) = 0. 

Because $(\talphak)_{k\in \N}$ as defined in \cref{eq:defi_alpha_k} is uniformly bounded in $\mW^{1,\infty}((0,T))$ with respect to $k\in \N$, we obtain uniform convergence on a subsequence to $\talphas \in \mW^{1,\infty}((0,T))$ for $k\rightarrow \infty$.

Since $(\tQk)_{k\in \N}$ converges locally uniformly and $( \talphak)_{k\in \N}$ converges uniformly on $(0,T)$ (recall that $\tQk = \Qkepsk - \talphak$), the sequence $(\Qkepsk)_{k\in \N}$ converges uniformly on compact sets too. By the choice of the sequence $\eps_{\tk}$ (compare \cref{eqn:subsequence1k}), we obtain $\Qkepsk - \Qk \rightarrow 0$ in $\mL^{\infty}_\loc((0,T)\times \R)$. These uniform convergences on compact sets imply uniform convergence of $(\Qk)_{k\in \N}$ on compact sets too. 
If we define $\Qs \coloneqq \lim_{k\rightarrow \infty} \Qk$, the relation $\tQs = \Qs - \talphas$ on \([0,T]\times\R\) holds.

Now, we prove that \(\tQs\) is a viscosity solution of the \HJ equation, and by \cref{theo:rel_HJ_cl}, this means that \(\partial_2 \tQs\) is an entropy solution of \cref{eq:local_c_l}. As $\talphas$ is only time dependent, this results in  \(\partial_2 \Qs\) being (also) an entropy solution of \cref{eq:local_c_l}, from which the claim follows.

\end{enumerate}

The next part of the proof can be viewed as a special case of Berge’s maximum theorem \cite{Berge1963}. Since our situation involves the particular subsequences introduced above, we provide a shorter proof that is limited to this setting.

\begin{enumerate}[label=(\alph*)]\setcounter{enumi}{1}
\item \textit{Convergence of optimizers under uniform convergence:}\label{item:convergence_optimizers} Let \(\Phi \in \mC^\infty([0,T]\times \R)\) be arbitrary but fixed. Let \((t_0,x_0)\in (0,T)\times \R\) be chosen such that \(\tQs -\Phi\) has a local (not necessarily strict) maximum in \((t_0,x_0)\).
    Let 
    \begin{align}\label{eqn:strict_local_maximum_test_function}
        \tilde \Phi(t,x) \coloneqq \Phi(t,x) + (t-t_0)^2 + (x-x_0)^2.
    \end{align} It holds that \(\partial_1 \tilde \Phi(t_0,x_0)  = \partial_1 \Phi(t_0,x_0) \)  \(\partial_2 \tilde \Phi(t_0,x_0)  = \partial_2 \Phi(t_0,x_0) \), and so $\tQs  - \tilde \Phi$ now has a strict local maximum in \((t_0,x_0).\) Defining $\overline{R} = \min\{1,t_0,T-t_0\},$ we obtain a sufficiently small neighborhood around \((t_{0},x_{0})\) still in \([0,T]\times \R\)  so that $\exists R \in (0,\overline R)$ such that
    \begin{align}
         \forall (t,x) \in B_{R}((t_0,x_0)) \! \setminus \! \{(t_0,x_0)\} \, : \, (\tQs -\tilde \Phi)(t,x) < (\tQs -\tilde \Phi)(t_0,x_0). \label{eqn:loc_max}
    \end{align} 
    Because $(t_0,x_0)$ is a strict local maximum, for any $r\in (0,R)$, there exists $\kappa_r \in\R_{>0}$ such that
    \begin{align}
        \forall (t,x) \in B_{R}((t_0,x_0)) \setminus B_{r}((t_0,x_0))  : \ (\tQs -\tilde \Phi)(t,x) \leq (\tQs -\tilde \Phi)(t_0,x_0) - \kappa_r. \label{eqn:glob_max}
    \end{align}
    Now, choose $I = (x_0-1,x_0+1)$ and, to apply \cref{eqn:subsequence1k}, a $K \in \N$ large enough that $I \subset (-K,K)$ and $k \geq K$.
    
    Because of the locally uniform convergence of $\tQk$ to $\tQs$, we can choose any $r\in (0,R)$ $K_r \in \N_{\geq K}$ large enough that
 \begin{align}\label{eqn:tQk-Qs_kappa}
        \|\tQk - \tQs\|_{L^\infty((0,T) \times I)} &\leq \tfrac{\kappa_r}{3} \ \forall k\in\N_{\geq K_r}.
        \end{align}

        Now, we estimate $\tQk - \tilde \Phi$ from above in $B_R \setminus B_r$. First, $\forall (t,x) \in B_{R}((t_0,x_0))\setminus B_r((t_0,x_0))$, we find---recalling that by definition of $\overline{R}$, we have $B_R(t_0,x_0) \subset (0,T)\times I$---that
    \begin{align}
        (\tQk  -\tilde \Phi)(t,x) &\leq \|\tQk  -\tQs\|_{\mL^\infty((0,T) \times I)} +(\tQs-\tilde \Phi)(t,x) \\
        & \leq (\tQs-\tilde \Phi)(t_0,x_0) - \tfrac{2\kappa_r}{3} \ \forall k\in\N_{\geq K_r},
    \end{align}
    with the last line implied by \cref{eqn:glob_max} and \cref{eqn:tQk-Qs_kappa}.
   From below, we obtain $\forall (t,x) \in B_r(t_0,x_0)$ (i.e., on the smaller ball around \((t_{0},x_{0})\)),
    \begin{align}
        (\tQk  -\tilde \Phi)(t_0,x_0) &\geq -\|\tQk  -\tQs\|_{\mL^\infty((0,T) \times I)} +(\tQs -\tilde \Phi)(t_0,x_0) \\
        & \geq (\tQs-\tilde \Phi)(t_0,x_0) - \tfrac{\kappa_r}{3}\ \forall k\in\N_{\geq K_r}.   \end{align}
     Thus, by continuity of $\tQk$, there is at least one point $(t_k,x_k) \in B_r(t_0,x_0)$ that is a global maximum of $(\tQk  -\tilde \Phi )|_{B_R(t_0,x_0)}$. Because this holds for any $r\in (0,R)$, we find for \(k\rightarrow\infty\) that the sequence of global maxima $(t_k,x_k)_{k\in \N}$ of $\big(\tQk -\tilde \Phi\big)|_{B_R(t_0,x_0)}$ converges to the global maximum  \((t_0,x_0)\) of \((\tQs - \tilde \Phi)|_{B_{R}(t_{0},x_0)}\).

     The same arguments hold for a local minimum.
\end{enumerate}
With this at hand, we can finally prove that $\tQk$ indeed converges to the viscosity solution of the \HJ equation. To show that $\tQs $ is indeed a viscosity solution of the \HJ equation, we investigate whether the sequence $\tQk $ 
satisfies the super-/subsolution property of \cref{defi:HJ_viscosity} in the limit.

\begin{enumerate}[label=(\alph*),resume]
\item \textit{Viscosity solutions of the \HJ equation:}
For the subsolution property, we need to prove the following inequality for all $\Phi \in \mC^\infty([0,T]\times\R)$:
\begin{align}\label{eqn:sub_solution_prove}
    \partial_1 \Phi(t_0,x_0) + \partial_2 \Phi(t_0,x_0)V\big(\partial_2 \Phi(t_{0},x_{0}),\partial_2 \Phi(t_0,x_0)\big) \leq 0,
\end{align}
where $(t_0,x_0)\in(0,T)\times\R$ is a local maximum of $\tQs -\Phi$.
As in the previous part of this proof, we change $\Phi$ to $\tilde \Phi$ as in \cref{eqn:strict_local_maximum_test_function} such that we have a strict local maximum at \((t_{0},x_{0})\) and the first derivatives of $\tilde \Phi$ (as defined in \cref{eqn:strict_local_maximum_test_function}) and $\Phi$ coincide at this point. In addition, without violating this property, we truncate $\tilde \Phi$ outside $(t_0-1,t_0+1)\times(x_0-1,x_0+1)$ in a smooth way so that $\partial_{22} \tilde \Phi \in \mL^{\infty}((0,T)\times \R)$. Recall that, since $\tQk$ is smooth, it is, as stated in \cref{eqn:nonlocal_HJ_solution}, a classical solution to
\begin{align}\label{eqn:nonlocal_HJ_solution_2nd}
   \partial_1 \tQk(t,x) + \partial_2 \tQk(t,x) V\big(\Wmk[\partial_2 \tQk](t,x),\Wpk[\partial_2 \tQk](t,x)\big) = 0 \ \text{ on } (0,T)\times\R.
\end{align}
Because \((t_0,x_0)\) is a strict local maximum of $\tQs - \tilde \Phi$ and $\tQk$ converges uniformly to $\tQs$ on compact sets, there exists for sufficiently large $K\in \N$ and sufficiently small $R \in \R$, a sequence $(t_k,x_k)_{k\in \N_{> K}} \subset B_R((t_0,x_0))$ for which $(t_k,x_k) \rightarrow (t_0,x_0)$ and $\tQk  - \tilde \Phi$ is locally maximal in $(t_k,x_k) \ \forall k \in \N_{> K}$. (This follows from the analysis in \cref{item:convergence_optimizers}, i.e., ``convergence of optimizers under uniform convergence.'') Because of the regularity of $\tQk$ and $\tilde \Phi$, by applying first-order optimality conditions, \(\forall k\in\N_{>K}\), we obtain that
\begin{align}
   \partial_1 \tQk(t_k,x_k) = \partial_1 \tilde \Phi(t_k,x_k) \qquad \wedge \qquad \partial_2 \tQk(t_k,x_k) = \partial_2 \tilde \Phi(t_k,x_k).
   \label{eq:partial_1_tQK_blah}
\end{align}
Inserting \cref{eq:partial_1_tQK_blah} in \cref{eqn:nonlocal_HJ_solution_2nd} at the point \((t_{k},x_{k})\) yields
\begin{align}
    \partial_1 \tilde \Phi(t_k,x_k) + \partial_2 \tilde \Phi(t_k,x_k) V\big(\Wmk[\partial_2\tQk](t_k,x_k),\Wpk[\partial_2 \tQk](t_k,x_k)\big) = 0.
\end{align}
Recalling that we can approximate
\(\gamma\in \mBV(\R_{<0})\) by a sequence of smooth functions \(\gamma_{\eps}\in \mC^{\infty}(\R_{<0};\R_{\geq 0})\cap\mW^{1,1}(\R_{<0};\R_{\geq 0})\) so that \(\gamma_{\eps}'\geqq 0 \text{ on } \R_{<0}\) and $\forall x\in \R_{<0}$,
\[
\lim_{\eps\rightarrow 0}\|\gamma_{\eps}-\gamma\|_{\mL^{1}((-\infty,x))}=0\ \wedge\ \lim_{\eps\rightarrow 0} \|\gamma_{\eps}\|_{\mL^{\infty}((-\infty,x))}= \|\gamma\|_{\mL^{\infty}((-\infty,x))}.
\]
We omit the details of this smoothing argument and assume for $\gammapk$, in addition to \cref{ass:general}, that $\gammapk\in\mC^{\infty}(\R_{<0};\R_{\geq 0})\cap\mW^{1,1}(\R_{<0};\R_{\geq 0})$ for the intermediate steps involving $(\gammapk)'$.

We then find, for the nonlocal impacts, that
\begin{align}
    \Wpk[\partial_2 \tQk](t_k,x_k)&\coloneqq \int_{x_k}^\infty \gammapk(x_k-y)\partial_2 \tQk (t_k,y) \dd y \\
    &= \int_{x_k}^\infty (\gammapk)'(x_k-y)\big(\tQk (t_k,y)-\tQk (t_k,x_k)\big) \dd y.
\end{align}
Because $\tQk (t_k,x_k) -\tilde \Phi(t_k,x_k) \geq \tQk (t,y) -\tilde \Phi(t,y)$, for all $r\in (0,R)$ and for all $(t,y) \in B_{R-r}(t_k,x_k) \subset B_R(t_0,x_0)$, we obtain the following:
\begin{align}
    \tilde \Phi(t,y) -\tilde \Phi(t_k,x_k) \geq \tQk (t,y) -\tQk (t_k,x_k).
\end{align}
Recalling that $(\gammapk)' \geq 0$, we find that
\begin{align}
    &\Wpk[\partial_2 \tQk ](t_k,x_k)\\
&\leq \int_{x_k}^{x_k+R-r} (\gammapk)'(x_k-y)\big(\tilde \Phi(t_k,y) -\tilde \Phi(t_k,x_k)\big) \dd y \\
    &\quad + \int_{x_k+R-r}^\infty (\gammapk)'(x_k-y)\Big(\tQk (t_k,y)-\tQk (t_k,x_k)\Big) \dd y \\
&\leq  \gammapk((r-R)^-)\big(\tilde \Phi(t_k,x_k+R-r) - \tilde \Phi(t_k,x_k) \big) + \int_{x_k}^{x_k+R-r} \!\!\!\!\!\!\!\!\!\!\!\!\!\!\gammapk(x_k-y)\partial_y \tilde \Phi(t_k,y) \dd y \\
    &\quad + \gammapk((r-R)^-)(R-r)\|\qn \|_{\mL^\infty(\R)}  + 2\|\gammapk\|_{\mL^ 1((-\infty,r-R))}\|\qn \|_{\mL^\infty(\R)}.
\end{align}
Note that because of the kernel monotonicity, we have $\|\gammapk\|_{\mL^\infty((-\infty,x))} = \gammapk(x^-) \ \forall x \in \R_{<0}$, so the approximation argument is still valid. Because of \cref{ass:general} and \cref{lem:kernel_vanishing_first_moments}, all terms except for the integral vanish for $k \rightarrow \infty$. We investigate the remaining term further and show that the integral approximates $\partial_2 \tilde \Phi(t_0,x_0)$:
\begin{align*}
    \bigg|\partial_2 \tilde \Phi(t_k,x_k) - \int_{x_k}^{x_k+R-r} \!\!\!\!\!\!\!\!\!\!\!\!\!\!\!\!\!\!\gammapk(x_k-y)\partial_y \tilde \Phi(t_k,y) \dd y \bigg| &=  \bigg|\int_{x_k}^{x_k+R-r} \!\!\!\!\!\!\!\!\!\!\!\!\!\!\!\!\gammapk(x_k-y) \int_{x_k}^y \!\!\! \partial_2^2 \tilde \Phi(t_k,s) \dd s \bigg)\dd y\bigg|  \\
    &\leq \|\partial_2^2 \tilde \Phi\|_{\mL^\infty((0,T)\times \R)} \int_{0}^{R-r} y\gammapk(-y)  \dd y.
\end{align*}
Because of the first moment property of the kernel in \cref{lem:kernel_vanishing_first_moments}, the latter integral vanishes for $k\rightarrow \infty$.

Altogether,  there exists a null sequence $(\sCkp)_{k\in \N}\subset \R$ such that
\begin{align}\label{eqn:est_Wpk}
    \Wpk[\partial_2 \tQk (t_k,\cdot)](x_k) \leq \partial_2 \tilde \Phi(t_k,x_k) + \sCkp.
\end{align}
Analogously for $\Wmk$, there exists a null sequence $(\sCkm)_{k\in \N}\subset \R$ such that
\begin{align}\label{eqn:est_Wmk}
    \Wmk[\partial_2 \tQk (t_k,\cdot)](x_k) \geq \partial_2 \tilde \Phi(t_k,x_k) + \sCkm.
\end{align}
Using $\partial_1 V \geqq 0$ and  $\partial_2 V \leqq 0$, we obtain
\begin{align}
    0 &= \partial_1 \tilde \Phi(t_k,x_k) + \partial_2 \tilde \Phi(t_k,x_k) V\big(\Wm[\partial_2\tQk (t_k,\cdot)](x_k),\Wp[\partial_2 \tQk (t_k,\cdot)](x_k)\big), \\
    \intertext{which can be estimated using \cref{eqn:est_Wpk} and \cref{eqn:est_Wmk}. From \cref{eq:partial_1_tQK_blah}, we obtain $\partial_2 \tilde \Phi(t_k,x_k) = \partial_2 \tQk(t_k,x_k)$ using \cref{eqn:Qx_q_equivalence} and the nonnegativity of the initial data (\cref{ass:general}). This results in $ \partial_2 \tilde \Phi(t_k,x_k) \geq 0$, and thus}
    0 &\geq \partial_1 \tilde \Phi(t_k,x_k) + \partial_2 \tilde \Phi(t_k,x_k) V\big(\partial_2\tilde \Phi(t_k,x_k) + \sCkm,\partial_2\tilde \Phi(t_k,x_k) + \sCkp\big).
\end{align}
Because $\tilde \Phi \in \mC^\infty((0,T)\times \R)$, $\sCkp \rightarrow 0$, $\sCkm \rightarrow 0$, and $(t_k,x_k) \rightarrow (t_0,x_0)$, we end up in the subsolution condition, that is,
\begin{align}
    \partial_1 \tilde \Phi(t_0,x_0) + \partial_2 \tilde \Phi(t_0,x_0) V\big(\partial_2\tilde \Phi(t_0,x_0),\partial_2\tilde \Phi(t_0,x_0)\big) \leq 0.
\end{align}

Because the first derivatives of $\tilde{\Phi}$ and $\Phi$ coincide at $(t_0,x_0) \in ((0,T)\times \R)$---see the discussion between \cref{eqn:strict_local_maximum_test_function} and \cref{eqn:loc_max}---we have thus demonstrated that \cref{eqn:sub_solution_prove} also holds. From this, we have that $ \tQs$ is a subsolution. In the same way, we can prove that \(\tQs\) is a supersolution, which implies that $\tQs$ is the unique viscosity solution of \cref{eq:H_J} (for the definition, see  \cref{defi:HJ_viscosity}). 
\end{enumerate}
We can now deduce from this uniform convergence of $\Qk$ to $\Qs$ on compact sets the weak-star convergence of the respective nonlocal solutions $(\qk)_{k\in \N}$.
\begin{enumerate}[label=(\alph*),resume]
    \item \textit{Weak-star convergence of nonlocal solutions and nonlocal terms:}
We know from \cref{theo:rel_HJ_cl} that
\[
\partial_{2}\tQs\in\mL^{\infty}((0,T)\times\R)
\]
is the (unique) entropy solution of \cref{eq:local_c_l}. Because \(\partial_{2}\tQs\equiv \partial_{2}\Qs\) on \((0,T)\times\R\), we can see, for test functions \(\phi\in \mC^{1}_{\text{c}}((0,T)\times \R)\)
and the sequence \(\qk\) as in \cref{eq:weak_star:convergence}, that
\begin{align}
    \iint \qk \phi \dd x \dd t = - \iint \Qk \phi_x \dd x  \dd t\ \xrightarrow{k\rightarrow \infty} - \iint \Qs \phi_x \dd x   \dd t= \iint \partial_2 \Qs \phi \dd x  \dd t.
\end{align}
Since $\mC^{1}_c((0,T)\times \R)$ is dense in $\mL^1((0,T)\times \R)$, we therefore obtain
\[
\forall g\in\mL^{1}((0,T) \times \R):\ \lim_{k\rightarrow\infty} \iint\big(\qk(t,\cdot)-\partial_{2}\Qs(t,\cdot)\big) g(t,x)\dd x \dd t =0,
\]
which is the proposed weak-star convergence.
So far, our analysis has been based on a selection criterion regarding subsequences. However, because we can repeat the same steps for any other subsequence (for which a subsequence exists that converges to the unique entropy solution), the claimed convergence holds for any sequence. 
The claimed convergence of the nonlocal terms \(\Wmk,\Wpk\) to the local entropy solutions follows from their definition, the weak convergence of \(\qk\), the continuity of \(\mL^{1}\) functions with respect to shifts in the argument, and the ``vanishing tail property'' of the nonlocal kernels as in \cref{ass:general}.
\end{enumerate}
\end{proof}

\begin{rem}[applications of the convergence result to one-sided kernels and velocity functions with sign-restricted derivatives]
   Note that \cref{theo:singular_limit_convergence} contains, as a byproduct, singular-limit results for what are mainly known in the literature as traffic flow models. Indeed, if we choose \(V\) so that \(\partial_{1}V\equiv 0\), we are in the situation of the nonlocal dynamics
   \begin{equation}
\partial_{t}q +\partial_{x}\big(\hat{V}(\Wp[q](t,x))q\big)=0\label{eq:nonlocal_one-sided}
   \end{equation}
   with \(\hat{V}\in \mW^{1,\infty}_{\loc}(\R): \hat{V}'\leqq 0\), and we find that \(q\) converges weak-star in \(\mL^{\infty}((0,T)\times\R)\) to the local entropy solution. Thus, we do not require the typical \(\mTV\)-bounds on the initial datum and compactness methods in \(\mL^{1}\) that are usually considered in \cite{coclite2022general,friedrich2024conservation,colombo2023nonlocal} and related results. The advantage of these results is, however, that one also obtains strong convergence in \(\mL^{1}\) to the local entropy solution, at least for the nonlocal operator \(\Wp[q]\). 
Only in recent results could the required \(\mTV\) bounds for the initial datum be dispensed by using compensated compactness \cite{coclite2025singular}; however, this approach is still restricted to the model class in which the velocity's derivative has a sign, as presented in \cref{eq:nonlocal_one-sided}, and it is more restrictive on the nonlocal kernels.

\end{rem}

\subsection{Nonlocal Approximation Result for Sign-Unrestricted Data}\label{subsec:sign_restricted}
In the previous results, we have demonstrated that classical nonlocal approximations looking backward and forward suffice to approximate entropy solutions to local conservation laws. However, on the nonlocal level, it was necessary to assume that the initial datum has a sign to obtain uniform \(\mL^{\infty}\)-bounds on the solution, enabling it to pass to the limit in the nonlocal kernel and the solution. 

By shifting the initial datum by a constant (typically the essential infimum), we can generalize our nonlocal approximation result to sign-unrestricted data with the drawback that a physical interpretation as in \cref{eq:nonlocal_c_l} is not obvious. This is detailed in the following theorem.

\begin{theo}[nonlocal approximation of scalar nonlinear conservation laws with sign-unrestricted data]\label{theo:sign_unrestricted_initial_datum}
Let \(\qn \in\mL^{\infty}(\R)\) and define
\(\R\ni \um\leq \essinf_{x\in\R}\qn (x)\). Further, let \(f \in \mW^{1,\infty}_\loc(\R)\) satisfy $s \mapsto \nicefrac{(f(s +\um) - f(\um))}{s} \ \forall s \in \R_{\neq 0}$ possess a locally Lipschitz extension on $\R$. Then, the local entropy solution \(q\in\mL^{\infty}((0,T)\times\R)\) of \cref{eq:local_c_l} can be approximated weak-star in \(\mL^{\infty}((0,T)\times\R)\) by the (unique) weak solution \(\qk\in\mC([0,T];\mL^{1}(\R))\cap \mL^{\infty}((0,T)\times\R)\) of the following nonlocal conservation law:
\begin{equation}
\begin{aligned}
    \partial_t q + \partial_x \big((q-\um) V(\Wmk[q-\um],\Wpk[q-\um]) \big) &= 0,&&(t,x)\in (0,T)\times\R,\\
    q(0,\cdot) &\equiv \qn, && (t,x)\in (0,T)\times\R,
\end{aligned}
\label{eq:nonlocal_sign_unrestricted}
\end{equation}
with \(V\) as in \cref{ass:general}, \(\tilde V(s) \coloneqq \nicefrac{(f(s +\um) - f(\um))}{s}\ \forall s\in \R\), and \(\Wmk\) and \(\Wpk\) as in \cref{notation}.
\end{theo}
\begin{proof}
    We only sketch the proof and consider the dynamics classically, without loss of generality, as the main argument consists of applying \cref{theo:singular_limit_convergence} on a properly reformulated (non)local conservation law.
Starting with the local conservation law \cref{eq:local_c_l} supplemented by an appropriate initial datum
    \begin{align*}
           \rho_t + (f(\rho))_x &= 0, \qquad \rho(0,\cdot)\equiv\qn,  && (t,x)\in (0,T)\times\R,
\intertext{we substitute \(\orho\equiv \rho-\um \) and because \(m\equiv \text{const.},\) \(\orho\) satisfies}
        \partial_{t}\orho + \partial_{x}\big(f(\orho+\um) -f(\um)\big)&=0,\qquad \orho(0,\cdot)\equiv\qn (x)-\um\geqq 0, &&(t,x)\in (0,T)\times\R,
\intertext{which, by definition of $\tilde V$ is equivalent to}
        \partial_{t} \orho+ \partial_{x}\big(\orho\,\tilde V(\orho)\big)&=0, \qquad \orho(0,\cdot)\equiv\qn (x)-\um\geqq 0, &&(t,x)\in (0,T)\times\R.
\intertext{    This local conservation law is then approximated nonlocally by the unique weak solution (as guaranteed in \cref{theo:existence_uniqueness_nonlocal}) \((\oqk)_{k\in\N}\) of}
        \partial_{t}\oq + \partial_{x}\big(\oq \,V(\Wmk[\oq],\Wpk[\oq])\big)&=0, \qquad \oq(0,\cdot)\equiv\rho_{\tn}(x)-\um\geqq 0, &&(t,x)\in (0,T)\times\R,
    \end{align*}
which, after the substitution \(q \equiv \oq+\um\), results in the nonlocal conservation law stated in \cref{eq:nonlocal_sign_unrestricted}. Because we know from \cref{theo:singular_limit_convergence} that \(\overline q_{\tk}\) converges to local entropy solutions for \(\oq(0,\cdot)=\qn -\um\) thanks to this initial datum being nonnegative, we also obtain the convergence of \(\qk\) with initial datum \(\qn \) as stated in \cref{eq:nonlocal_sign_unrestricted}.
\end{proof}

\section{Rate of Convergence of the Nonlocal Approximation}\label{sec:convergence_order}
For simplicity, we restrict ourselves in this section to convex fluxes $f$. By analogous arguments, we can obtain the same results for concave fluxes.
We recall the following classical sup-/inf-convolution regularization, a standard
tool in the theory of viscosity solutions and in the study of
semiconvex/semiconcave approximations; see, e.g.,
\cite[Sec.~3.5]{CannarsaSinestrari2004}.
\begin{defi}[inf-/sup-convolutions and smooth approximations]\label{defi:inf_sup_convolutions} For $Q \in \mW^{1,\infty}_\loc((0,T)\times \R)$ with $0 \leqq \partial_2 Q \leqq \qmax$, we define, for \(\eps\in\R_{>0}\) and \((t,x)\in (0,T)\times\R\),
\begin{align}
 \Qepsp(t,x) \coloneqq \sup_{y\in\R} \Big\{Q(t,y) - \tfrac{(x-y)^2}{2\eps}\Big\} , \qquad \Qepsm(t,x) \coloneqq \inf_{y\in\R} \Big\{Q(t,y) + \tfrac{(x-y)^2}{2\eps}\Big\},\label{eq:Q_pm_eps}
\end{align}
and for the standard mollifier $\phi_\teps$, we define ``spatially'' smooth approximations thereof, that is, for \(t\in[0,T]\) and \(\eps\in\R_{>0}\),
\begin{align}
    \Ueps(t,\cdot) \coloneqq \Qepsp(t,\cdot) * \phi_\teps, \qquad \Leps(t,\cdot) \coloneqq \Qepsm(t,\cdot) * \phi_\teps.\label{defi:U_eps}
\end{align}
\end{defi}

The well-definedness of the suprema follows from the Lipschitz continuity of $Q$
\cite[Lem.~3.5.2]{CannarsaSinestrari2004}.
For a result on approximating the solution to the \HJ equation, we establish the following uniform convergence property and bound on the nonlocal operators applied to the smooth approximations.

\begin{lem}[properties of the smooth approximations of $Q$]\label{lem:smooth_approx}
Given \cref{defi:inf_sup_convolutions}, we obtain the following approximation for \(\eps\in\R_{>0}\):
\begin{equation}
    \|\Ueps - Q\|_{\mL^\infty((0,T) \times \R)} \leq \sCn\eps, \  \|Q-\Leps\|_{\mL^\infty((0,T) \times \R)} \leq \sCn\eps, \label{eq:U_eps_Q_approximation}
\end{equation}
with $\sCn = \qmax+\nicefrac{\qmax^2}{2}$. We also have the following approximations for the application of the nonlocal operators \(\Wmk,\Wpk\) as in \cref{notation}, for \((t,x)\in[0,T]\times\R\) and $\eps \in (0,\nicefrac{1}{q_{\tmax}})$:
\begin{equation}
\begin{aligned}
\Wpk[\partial_{2} \Ueps](t,x)\ge \partial_x \Ueps(t,x)-\tfrac{\mpk}{\eps},
\qquad
\Wmk[\partial_{2} \Ueps](t,x)\le \partial_x \Ueps(t,x)+\tfrac{\mmk}{\eps},\\
    \Wpk[\partial_{2} \Leps](t,x)\le \partial_x \Leps(t,x)+\tfrac{\mpk}{\eps},
\qquad
 \Wmk[\partial_{2} \Leps](t,x)\ge \partial_x \Leps(t,x)-\tfrac{\mmk}{\eps}.
\end{aligned}
\label{eq:Wmpk_partial_2_U_eps}
\end{equation}
with the ``truncated'' first moments 
\begin{align}\label{eqn:truncated_first_moments}
    \mmk \coloneqq \int_0^\infty \min\{1,x\} \gammamk(x) \dd x, \quad \text{and} \quad \mpk \coloneqq \int_0^\infty \min\{1,x\} \gammapk(-x) \dd x.
\end{align}
\end{lem}
\begin{proof}
We first show the approximation result in \cref{eq:U_eps_Q_approximation} and, without loss of generality, we focus only on the result on the left concerning \(\Ueps -Q\); the result on the right follows analogously.

    The sup-/inf-convolutions maintain the monotonicity and Lipschitz continuity of $Q$ induced by the boundedness and nonnegativity of $\qn $ (\cref{ass:general}).
    For \((t,x,y)\in[0,T]\times\R^{2}\) and \(\eps\in\R_{>0}\), it holds that
    \begin{equation}
    \begin{aligned}
        Q(t,y) - \tfrac{(x-y)^2}{2\eps} &=Q(t,x)+\int_{x}^{y}\partial_{2}Q(t,z)\dd z\\
        &\leq  Q(t,x) + \qmax |x-y| - \tfrac{(x-y)^2}{2\eps} \leq Q(t,x) + \tfrac{\qmax^2 \eps}{2},
    \end{aligned}
    \label{eq:lower_upper_bound_Q_eps_plus}
    \end{equation}
    where the first inequality holds because of the $\qmax$-Lipschitz continuity of \(Q\), and the second was obtained by maximizing the concave function $\R\ni s \mapsto \qmax s - \nicefrac{s^2}{2\eps}$ over $\R_{\geq 0}$,  leading to the maxima $s^* = \pm \qmax\eps$ and thus the given maximal value \(\nicefrac{\qmax^{2}\eps}{2}\). 
    By construction, $\Qepsp$ is bounded from below by $Q$, and by taking the supremum with respect to \(y\in\R\) of the left-hand side in the latter inequality, \cref{eq:lower_upper_bound_Q_eps_plus} leads, for \((t,x)\in [0,T]\times\R\), to 
    \begin{align}
        Q(t,x) \leq \Qepsp(t,x) \leq Q(t,x) + \tfrac{\qmax^2 \eps}{2}.\label{eq:Q_Q_+_eps}
    \end{align}
Now, we add a zero and write, for \((t,x)\in[0,T]\times\R\),
\begin{align}
    \Ueps (t,x)-Q(t,x)&=\Ueps (t,x)-\Qepsp (t,x)+\Qepsp (t,x)-Q(t,x)\label{eq:U_eps_Q}\\
    &\leq \|\Ueps -\Qepsp \|_{\mL^{\infty}((0,T)\times\R)}+\|\Qepsp -Q\|_{\mL^{\infty}((0,T)\times\R)}
    \intertext{and thanks to the $\qmax$-Lipschitz continuity of $\Qepsp$ with respect to the second component, the convolution with the standard mollifier satisfies
    $\|\Ueps  - \Qepsp\|_{\mL^\infty((0,T) \times \R)} \leq \qmax\eps$, which, together with \cref{eq:Q_Q_+_eps} leads to}
    \Ueps (t,x)-Q(t,x)&\leq \qmax\eps+\tfrac{\qmax^{2}\eps}{2}.
\end{align}
    Because the right-hand side is independent of \((t,x)\in[0,T]\times\R\), we obtain \cref{eq:U_eps_Q_approximation} from ``above.'' From below, the estimate is almost identical. The derivations for $\Leps$ are carried out in a similar way, concluding the proof of the estimate in \cref{eq:U_eps_Q_approximation}.

It remains to show \cref{eq:Wmpk_partial_2_U_eps}.
As 
    $\R\ni x \mapsto \Qepsp(t,x) + \nicefrac{x^2}{2\eps}= \sup_{y\in\R} \big\{Q(t,y) - \nicefrac{(y^2 - 2xy)}{2\eps}\big\} $
    is a supremum over affine functions, it is convex \cite[Thm.~5.5.]{Rockafellar1970}, so we obtain 
    $\partial_x^2 \Big(\Qepsp(t,x) + \nicefrac{x^2}{2\eps}\Big) \geq 0$ in a distributional sense. In particular, it holds for all $(t,x) \in (0,T)\times \R$ after applying any standard mollifier \(\phi_{\eps}\in\mC^{\infty}(\R;\R_{\geq0})\). Integrating in space yields
    \begin{align*}
        \forall x\in\R:\ \int_{\R}\big(\Qepsp (t,y)+\tfrac{y^{2}}{2\eps}\big)\phi_{\eps}''(x-y)\dd y&\geq 0,
        \intertext{which, after performing integration by parts twice in the latter term, becomes}
        \forall x\in\R:\ \int_\R \Qepsp(t,y)\phi''_\teps(x-y) + \tfrac{1}{\eps}\phi_\teps(x-y) \dd y &\geq 0,
        \intertext{and because  \(\|\phi_\teps\|_{\mL^1(\R)} = 1\) and derivatives commute with convolutions (see \cite[Prop.~4.20]{brezis}), this yields}
        \forall x\in\R:\ \partial_{x}^{2}\int_\R \Qepsp(t,y)\phi_\teps(x-y) \dd y &\geq -\tfrac{1}{\eps}.
    \end{align*}
    The left-hand side is the definition of \(\Ueps \) in \cref{defi:U_eps}, so we obtain \(
     \partial_{2}^{2}\Ueps (t,\cdot) \geqq -\tfrac{1}{\eps}.\)
    From this, it follows by the mean value theorem that, \(\forall x\in\R\),
     \begin{equation}
     \begin{aligned}
          \partial_y \Ueps(t,y) &\geq \partial_x \Ueps(t,x) - \tfrac{y-x}{\eps} &&\forall y \geq x  \\
        \partial_y \Ueps(t,y) &\leq \partial_x \Ueps(t,x) + \tfrac{x-y}{\eps} &&\forall y \leq x. 
             \end{aligned}\label{eqn:bounds_bla}
    \end{equation}
Now, $\partial_y \Ueps(t,y),\partial_x \Ueps(t,y) \in [0,\qmax] \ \forall (t,x) \in [0,T]\times \R$ because the monotonicity and  $\qmax$-Lipschitz continuity of $\Qepsp$ are preserved for $\Ueps$. Combining this with the assumption that $\nicefrac{1}{\eps} \geq {\qmax}$, we can then bound the terms in \cref{eqn:bounds_bla} further:
     \begin{align}
        \partial_y \Ueps(t,y) &\geq \partial_x \Ueps(t,x) - \tfrac{\min\{1,y-x\}}{\eps} &&\forall y \geq x \label{eq:partial_2_U_eps_lower_bound}\\
        \partial_y \Ueps(t,y) &\leq \partial_x \Ueps(t,x) + \tfrac{\min\{1,x-y\}}{\eps} &&\forall y \leq x. \label{eq:partial_2_U_eps_upper_bound}
    \end{align}
Thus, by applying the definition of \(\Wpk\) in \cref{notation}, we find that for \((t,x)\in [0,T]\times\R\),
\begin{align*}
    \Wpk[\partial_{2}\Ueps ](t,x)&=\int_{x}^{\infty} \gammapk(x-y)\partial_{2}\Ueps (t,y)\dd y,
    \intertext{and because \(\gammapk\geqq 0\) thanks to \cref{ass:general} and \cref{eq:partial_2_U_eps_lower_bound} as \(y\geq x\),}
    \Wpk[\partial_{2}\Ueps ](t,x)&\geq \int_{x}^{\infty} \gammapk(x-y)\partial_{2}\Ueps (t,x)\dd y-\int_{x}^{\infty} \gammapk(x-y)\tfrac{\min\{1,y-x\}}{\eps}\dd y.
    \intertext{Using \cref{ass:general} and \cref{eqn:truncated_first_moments} leads to}
    \Wpk[\partial_{2}\Ueps ](t,x)&\geq \partial_{2}\Ueps (t,x)-\tfrac{\mpk}{\eps}.
\end{align*}
    This is indeed the first part of \cref{eq:Wmpk_partial_2_U_eps}; the upper bound on \(\Wmk[\partial_{2}\Ueps ]\) can be derived analogously by taking advantage of \cref{eq:partial_2_U_eps_upper_bound}. The    
    remaining inequalities in \cref{eq:Wmpk_partial_2_U_eps} for $\partial_{2}\Leps$ can also be obtained analogously.
\end{proof}

The final ingredient needed for the proof of the rate-of-convergence theorem is
the following lemma. It shows that the smooth approximations \(\Ueps\) and
\(\Leps\) are, respectively, sub- and supersolutions of the local
\HJ equation, and it quantifies their defect with respect to the
nonlocal \HJ equation.

\begin{lem}[sub-/supersolution property for $\Ueps$ and $\Leps$] \label{lem:error_nonlocal_H_J}
The mappings $\Ueps, \Leps:[0,T]\times\R\rightarrow\R$ given in \cref{defi:U_eps} are sub-/supersolutions of the \HJ equation \cref{eq:H_J} with initial data $Q^\eps_+(0,\cdot)* \phi_\teps$ and $Q^\eps_-(0,\cdot)* \phi_\teps$, respectively, in the sense of \cref{defi:HJ_viscosity}. Furthermore, the following inequalities hold a.e.\ in \(t\in(0,T)\) and \(x\in\R\)
\begin{align}
    \partial_1 \Ueps(t,x) + V\big(\Wmk[\partial_2 \Ueps](t,x),\Wpk[\partial_2 \Ueps](t,x)\big) \partial_2 \Ueps(t,x) &\leq \sCe\tfrac{\mmk+\mpk}{\eps},\\
        \partial_1 \Leps(t,x) +  V\big(\Wmk[\partial_2 \Leps](t,x),\Wpk[\partial_2 \Leps](t,x)\big)\partial_2 \Leps(t,x) &\geq -\sCe\tfrac{\mmk+\mpk}{\eps},
\end{align}
with $\sCe = \qmax\|\nabla V\|_{\mL^\infty((\qmin-1,\qmax+1)^2)}$ and $\mmk,\mpk$ as in \cref{eqn:truncated_first_moments}.
\end{lem}
\begin{proof}
    For a given $\Phi \in \mC^{\infty}((0,T)\times \R)$, let $(t_0,x_0) \in (0,T)\times \R$ be a local maximum of $\Qepsp - \Phi$. If there is no local maximum of $\Qepsp - \Phi$ on $(0,T)\times \R$, then there is nothing to show for this specific $\Phi$. By construction, there exists a $y_0$ such that $\Qepsp(t_0,x_0) = Q(t_0,y_0) - \tfrac{(x_0-y_0)^2}{2\eps}$, and we have, $\forall (t,x,y) \in (0,T) \times \R^2$, that
    \begin{align}
        \Qepsp(t,x) \geq Q(t,y) - \tfrac{(x-y)^2}{2\eps}.
    \end{align}
    After choosing $x = x_0+y-y_0$, this yields
    \(
       Q(t,y) \leq  \Qepsp(t,x_0+y-y_0) + \tfrac{(x_0-y_0)^2}{2\eps}
    \).  Now, subtracting $\Psi(t,y) \coloneqq \Phi(t,x_0+y-y_0) + \tfrac{(x_0-y_0)^2}{2\eps}$
    on both sides results in
     \begin{align}
       Q(t,y) - \Psi(t,y) \leq  \Qepsp(t,x_0+y-y_0) - \Phi(t,x_0+y-y_0).
    \end{align}
    By construction, $\Qepsp - \Phi$ is locally maximal in $(t_0,x_0)$; thus we have, for $(t,y)$ in the vicinity of $(t_0,y_0)$,
    \begin{align}
        \Qepsp(t,x_0+y-y_0) - \Phi(t,x_0+y-y_0) \leq \Qepsp(t_0,x_0) - \Phi(t,x_0).
    \end{align}
    We then obtain
\begin{align}
       Q(t,y) - \Psi(t,y) \leq  \Qepsp(t_0,x_0) - \Phi(t,x_0) &= Q(t_0,y_0) - \tfrac{(x_0-y_0)^2}{2\eps} - \Phi(t,x_0) \\
       &= Q(t_0,y_0)-\Psi(t_0,y_0)
    \end{align}
    and thus $Q-\Psi$ has a local maximum at $(t_0,y_0)$. Since $Q$ is a subsolution, it holds that
    \(
        \partial_1 \Psi(t_0,y_0) + f(\partial_2 \Psi(t_0,y_0)) \leq 0,
    \)
    and because the first derivatives of $\Psi(t_0,y_0)$ and $\Phi(t_0,x_0)$ coincide, we obtain 
        \(
        \partial_1 \Phi(t_0,x_0) + f(\partial_2 \Phi(t_0,x_0)) \leq 0.
    \)
    Thus, $\Qepsp$ is also a subsolution. According to \cite[Prop.~1.9.]{bardi2009optimal} and Rademacher's theorem, $\Qepsp$ satisfies 
    \begin{align}
        \partial_1 \Qepsp + f\big(\partial_2 \Qepsp\big) \leqq 0 \ \text{a.e.}\label{eqn:strong_HJ}
    \end{align}
    To obtain the subsolution property of $\Ueps$, we make use of the convexity of $f$:
    \begin{align}
        f(\partial_2 \Ueps(t,x)) &= f\Big( \big(\partial_2 \Qepsp (t,\cdot) * \phi_\teps\big)(x)\Big) \leq \Big(f(\partial_2 \Qepsp (t,\cdot) ) * \phi_\teps\Big)(x) \\
        \partial_t \Ueps(t,x) &= \partial_t \Big(\Qepsp(t,\cdot) * \phi_\teps\Big)(x) =  \Big(\partial_t \Qepsp(t,\cdot) * \phi_\teps\Big)(x).
    \end{align}
    Thus, by convolution of \cref{eqn:strong_HJ} with $\phi_\teps$, we obtain
        \begin{align}
        \partial_1 \Ueps + f(\partial_2 \Ueps) \leqq 0 \ \text{a.e.} \label{eq:super_sol_U}
    \end{align}
    and consequently---again by \cite[Prop.~1.9.]{bardi2009optimal}---$\Ueps$ is a subsolution. 

For the claimed inequalities, we obtain from $\partial_1 V \geq 0$, $\partial_2 V \leq 0$, $V\in \mW^{1,\infty}((\qmin,\qmax)^2)$ and \cref{lem:smooth_approx} the following:
    \begin{equation}
        V(\Wmk[\partial_2 \Ueps](t,x),\Wpk[\partial_2 \Ueps](t,x)) \leq V\Big(\partial_x \Ueps(t,x)+\tfrac{\mmk}{\eps},\partial_x \Ueps(t,x)-\tfrac{\mpk}{\eps}\Big).
    \end{equation}
Choosing $k$ large enough that $\mmk,\mpk \leq \eps$, we find that
    \begin{multline}
        V(\Wmk[\partial_2 \Ueps](t,x),\Wpk[\partial_2 \Ueps](t,x)) \leq V\big(\partial_x \Ueps(t,x),\partial_x \Ueps(t,x)\big) \\+ \|\partial_1 V\|_{\mL^\infty((\qmin-1,\qmax+1)^2)}\tfrac{\mmk}{\eps}  
        + \|\partial_2 V\|_{\mL^\infty((\qmin-1,\qmax+1
        )^2)}\tfrac{\mpk}{\eps}.
    \end{multline}
    Multiplying both sides by $\partial_2 \Ueps$, which satisfies $0 \leqq \partial_2 \Ueps \leqq \qmax$, we obtain
    \begin{align}
   \partial_x \Ueps(t,x) V\big(\Wmk[\partial_2 \Ueps](t,x),\Wpk[\partial_2 \Ueps](t,x)\big)
\leq
f(\partial_x \Ueps(t,x))
+
\sCe \tfrac{\mmk+\mpk}{\eps}.
\end{align}  
    This, together with \cref{eq:super_sol_U}, results in the claimed inequality for $\Ueps$. The arguments for $\Leps$ are analogous.
\end{proof}

This now allows us to combine the preceding lemmas and prove the main
rate-of-convergence result for the nonlocal approximation.

\begin{theo}[rate of convergence of the nonlocal approximation]\label{theo:convergence_order} Let $\Qk$ and $Q$ be antiderivatives of solutions $\qk$ to the nonlocal conservation law in \cref{eq:nonlocal_c_l} and the entropy solution \(q\)  of \cref{eq:local_c_l}, respectively.
Then, it holds for sufficiently large \(k\in\N\) that
\begin{align}
    \|\Qk - Q\|_{\mL^{\infty}((0,T)\times \R)} \leq \sCz \big(\mmk+\mpk\big)^{\nicefrac{1}{2}}
\end{align}
with $\sCz\coloneqq 2 \qmax\big(T\big(2 + \qmax\big)\|\nabla V\|_{\mL^{\infty}((\qmin,\qmax)^2)}\big)^{\nicefrac{1}{2}}$ and $\mmk$ and $\mpk$ as in \cref{eqn:truncated_first_moments}. In particular, it holds for every large enough \(k\in\N\) that
\begin{align}
    \|\qk- \qs\|_{\mL^\infty((0,T);\mW^{-1,\infty}(\R))} \leq \sCz \, \big(\mmk+\mpk\big)^{\nicefrac{1}{2}}.
\end{align}

\end{theo}

\begin{proof}
    To prove the convergence rate, we need to shift the derived $\Ueps$ and $\Leps$ to obtain upper and lower bounds for $\Qk$. Recall from \cref{lem:smooth_approx} that $\sCn \coloneqq \qmax + \tfrac{\qmax^2}{2}$, and for the sake of simplicity, we define
\begin{align*}
    \sCkeps \coloneqq \qmax\|\nabla V\|_{\mL^{\infty}((\qmin,\qmax)^2)}\tfrac{\mmk+\mpk}{\eps}. \quad 
\end{align*}
    We shift $\Ueps$ and $\Leps$ as follows:
    \begin{align}
        \uUeps \coloneqq \Ueps - \sCkeps t - \sCn \eps, \qquad \oLeps \coloneqq \Leps + \sCkeps t + \sCn \eps.
    \end{align}
By \cref{lem:smooth_approx}, $\Ueps(0,x) \leq Q(0,x) + \sCn \eps and \Leps(0,x) \geq Q(0,x) - \sCe \eps \ \forall x$, and we find, as a result of the shift by $\pm \sCe \eps$ for the initial datum, that
\(
    \uUeps(0,x) \leq Q(0,x) = \Qk(0,x) \leq \oLeps(0,x).
\)
Then, because of the linear shift in time \cref{lem:error_nonlocal_H_J}, we have
\begin{align}
    \partial_1 \uUeps + \partial_2 \uUeps V(\Wmk[\partial_2 \uUeps],\Wpk[\partial_2 \uUeps]) &\leqq 0.
\end{align}
By the comparison principle, which can be shown by considering the dynamics of $m(t) \coloneqq \sup_{x\in \R}\big(\uUeps(t,x) - \Qk(t,x)\big)$, the nonnegativity of the kernels $\gammamk$ and $\gammapk$, the monotonicity structure of $V$ (i.e., $\partial_1 V \geqq 0$, $\partial_2 V \leqq 0$), and the order structure in the initial datum, we can show that $\uUeps \leqq \Qk$. 
Again, by analogous arguments, we obtain $\Qk \leqq \oLeps$. By the definition of $\uUeps$ and the order principle of $Q$, we have
\(
    \Ueps - \sCkeps t - \sCn \eps + Q \leqq \Qk + Q
\)
and thus
\begin{align}
    Q-\Qk &\leq |Q-\Ueps| + \sCkeps t + \sCn \eps. \\
    \intertext{By \cref{lem:smooth_approx}, $t\in (0,T)$ and the definitions of $\sCkeps$, this becomes}
    Q-\Qk& \leq 2\sCn \eps + T \sCe \tfrac{\mmk + \mpk}{\eps}.
\end{align}
Because we obtain a similar result by considering $\overline \Leps$, we have
\begin{align}
    \|Q-\Qk\|_{\mL^\infty((0,T)\times \R)} \leq  2\sCn \eps + T\sCe\tfrac{\mmk + \mpk}{\eps}.
\end{align}
By minimizing the right-hand side over $\eps\in \R_{> 0}$, this leads to \[\eps^{\ts} = \Big(\tfrac{T\sCe(\mmk + \mpk)}{2\sCn}\Big)^{\nicefrac{1}{2}},\] which satisfies $\varepsilon \leq \qmax^{-1}$ for all $k\in \N$ large enough and by recalling the definitions of $\sCn$ and $\sCe$, that is, $\sCn \coloneqq \qmax + \tfrac{\qmax^2}{2},\sCe \coloneqq \qmax\|\nabla V\|_{\mL^\infty((\qmin,\qmax)^2)},$ the claim regarding the rate of the uniform convergence of $\Qk$ holds. 

 Since
\(
\partial_2(\Qk-Q)=\qk-\qs
\)
in the sense of distributions, we obtain for every
\(\phi\in W^{1,1}(\R)\) and a.e.\ in \(t\in(0,T)\) that
\begin{align}
\bigg|\int_\R(\qk-\qs)(t,x)\phi(x)\,\dd x\bigg|
&=
\bigg|\int_\R(\Qk-Q)(t,x)\partial_x\phi(x)\,\dd x\bigg|.
\intertext{Because \(\|\phi\|_{W^{1,1}(\R)}\leq1\), we find that}
\bigg|\int_\R(\qk-\qs)(t,x)\phi(x)\,\dd x\bigg|
&
\leq \|\phi\|_{W^{1,1}(\R)}\|\Qk-Q\|_{L^\infty((0,T)\times \R)}.
\end{align}
Taking the supremum over such \(\phi\) with \(\|\phi\|_{W^{1,1}(\R)} \leq 1\) 
yields the claim.
\end{proof}

\begin{rem}[convergence order for the ``classical'' scaling law for the nonlocal kernels]
    For $\gammam \in \mL^{1}(\R_{>0}) \cap \mL^\infty(\R_{>0})$ and $\gammap \in \mL^{1}(\R_{<0}) \cap \mL^\infty(\R_{<0})$ satisfying \cref{ass:general} with bounded first moments (i.e., $\|x \mapsto x \gammam(x)\|_{\mL^1(\R_{>0})} \eqqcolon \mm$), $\|x \mapsto x \gammap(-x)\|_{\mL^1(\R_{>0})} \eqqcolon \mp$ defines the kernel sequences $\gammamk \equiv k \gammam(k \, \cdot), \gammapk \equiv k \gammap(k \, \cdot) \ \forall k\in\N$. Then, we obtain, by a change of variable in the integration, $\mmk = k^{-1}\mm$ and $\mpk = k^{-1}\mp$. Then, from \cref{theo:convergence_order}, we find that for large enough \(k\in\N\)
    \begin{align}
    \|\Qk - Q\|_{\mL^{\infty}((0,T)\times \R)} \leq \tilde{\mathsf C}_{\tz}  \, k^{-\nicefrac{1}{2}},
\end{align}
with $\tilde{\mathsf C}_{\tz}\coloneqq 2 \qmax \Big(T\big(\mm+\mp\big)\big(2 + \qmax\big)\|\nabla V\|_{\mL^{\infty}((\qmin,\qmax)^2)}\Big)^{\nicefrac{1}{2}}$.
\end{rem}
    

\end{document}